\documentclass{elsarticle}
\newcommand{\f}{\operatorname}
\usepackage{lineno,hyperref}
\usepackage{graphicx}
\usepackage{amsmath}
\usepackage{amssymb}
\usepackage{amsthm}
\usepackage{multirow}
\usepackage{graphicx} 
\usepackage{blindtext}
\usepackage{xcolor}

\newcommand{\R}{\mathbb{R}}
\newcommand{\N}{\mathbb{N}}

\newtheorem{theorem}{Theorem}[section]
\newtheorem{corollary}[theorem]{Corollary}

\newtheorem{proposition}[theorem]{Proposition}

\theoremstyle{definition}
\newtheorem{definition}[theorem]{Definition}
\modulolinenumbers[5]

\begin{document}

\begin{frontmatter}

\title{Posterior properties of the Weibull distribution for censored data}
\tnotetext[mytitlenote]{Posterior properties of the Weibull distribution for censored data}

\author{Eduardo Ramos$^1$ and Pedro L. Ramos$^1$}

\ead{eduardoramos@usp.br}

\address{$^1$Institute of Mathematical and Computer Sciences, University of S\~ao Paulo, S\~ao Carlos, S\~ao Paulo, Brazil}

\begin{abstract}
The Weibull distribution is one of the most used tools in reliability analysis. In this paper, assuming a Bayesian approach, we propose necessary and sufficient conditions to verify when improper priors lead to proper posteriors for the parameters of the Weibull distribution in the presence of complete or right-censored data. Additionally, we proposed sufficient conditions to verify if the obtained posterior moments are finite. These results can be achieved by checking the behavior of the improper priors, which are applied in different objective priors to illustrate the usefulness of the new results. As an application of our theorem, we prove that if the improper prior leads to a proper posterior, the posterior mean, as well as other higher moments of the scale parameter, are not finite and, therefore, should not be used.
\end{abstract}

\end{frontmatter}

\section{Introduction}

The Weibull distribution plays a central role in reliability analysis as one of the most important generalizations of the exponential distribution. Introduced by Weibull \cite{weibull1951statistical}, his distribution has been used in many applications as baseline distribution. Let $X$ be a non-negative random variable with the Weibull distribution given by 
\begin{equation}\label{pdfgamma}
f(x\,| \eta,\beta)=\beta\eta^\beta (\eta x)^{\beta-1}e^{-\eta^\beta x^\beta} ,
\end{equation}
where $\eta>0$ and $\beta>0$ are unknown scale and shape parameters, respectively.

The maximum likelihood estimators, as well as other frequentist methods of inference for Weibull distribution, are standard in the reliability literature \cite{teimouri2013comparison}. From the Bayesian point of view, a prior distribution must be chosen for the parameters. Objective priors have already been derived for the Weibull distribution in the literature considering complete data, for instance, Sun \cite{sun1997note} have discussed many objective priors for the Weibull model such as the Jeffreys prior \cite{jeffreys1946theory}, reference priors \cite{bernardo1979a, bernardo2005, berger2015}, matching priors \cite{tibshirani1989}, and Moala et al. \cite{moala2009note} considered these priors as well as the maximum data information (MDI) prior \cite{zellner1,zellner2} for the reliability function. Some recent derivations of objective priors for related functions of Weibull distribution can be seen in \cite{lee2015noninformative,kang2017noninformative,lee2017objective}.

A major drawback related to the use of objective priors is that they are improper and could lead to improper posteriors. Northrop and Attalides \cite{northrop2016} argued that ``$\ldots$ there is no general theory providing simple conditions under which improper prior
yields a proper posterior for a particular model, so this must be investigated case-by-case". 
In this paper, for the Weibull distribution, we overcome this problem by providing sufficient and necessary conditions to check when a wide class of posterior distributions is proper. We then proceeded to use such theorem to investigate whether these improper priors lead to proper posterior distributions. Another important issue is related to the computation of the posterior moments since even if the posterior distribution is proper, the related posterior moments (posterior mean, posterior variance, among others) can be infinite. Thus, for the Weibull model, we shall also discuss sufficient conditions that lead the posterior moments to be proper. Based on our findings, further researches will be able to check if the obtained posterior is proper or improper easily, as well as the finiteness of the posterior moments analyzing the behavior of the improper prior directly. Although the proof for the posterior distribution has been considered for complete data, there is no similar proof for random censoring in the literature. Hence our results allow the use of the objective Bayesian analysis for the Weibull distribution in the presence of censoring. The new results are applied in different objective priors such as  independent uniform priors, Jeffreys' first rule \cite{kass1996selection}, Jeffreys' prior  \cite{jeffreys1946theory}, maximal data information (MDI) prior \cite{zellner1, zellner2} and reference priors \cite{bernardo1979a, bernardo2005, berger2015}.

The remainder of this paper is organized as follows. Section 2 introduces essential results that will be used to prove the new theorem. Section 3 presents a theorem that provides necessary and sufficient conditions for the posterior distributions to be proper and also sufficient conditions to check if the posterior moments of the parameters are finite. Section 4 presents the applications of our main theorem in different objective priors. Finally, Section 5 summarizes the study.

\section{Posterior distribution}

Let $X_1,\ldots,X_n$ be an realization of an independent and identically distributed sample where $X\sim$ Weibull$(\eta,\beta)$. Moreover, suppose that the $i$th  individual has a lifetime $X_i$ and a censoring time $C_i$, additionally, the random censoring times $C_i$s are independent of
$T_i$s and their distribution does not depend on the parameters, then the data set is $(t_i,\delta_i)$, where $t_i=\min(X_i,C_i)$ and $\delta_i=I(T_i\leq C_i)$ is an indicator function of the presence of censoring. Considering a prior distribution $\pi(\eta,\beta)$, the joint posterior distribution for $(\eta,\beta)$ is given by 
\begin{equation}\label{principal}
\pi(\eta,\beta|\boldsymbol{x})=\frac{\pi(\eta,\beta)}{d(\boldsymbol{x})} \eta^{m\beta}\beta^{m}\left(\prod_{i=1}^n x_i^{\delta_i\beta}\right)\exp \left\{-\sum_{i=1}^n x_i^\beta \eta^{\beta}\right\},
\end{equation}
where $m=\sum_{i=1}^{n}\delta_i$, and  $d(\boldsymbol{x})$ is a normalized constant in the form 
\begin{equation}\label{posteriord2}
d(\boldsymbol{x})=\int_{0}^{\infty}\int_{0}^{\infty}\pi(\eta,\beta) \eta^{m\beta}\beta^{m}\left(\prod_{i=1}^n x_i^{\delta_i\beta}\right)\exp \left\{-\sum_{i=1}^n x_i^\beta \eta^{\beta}\right\}d\eta d\beta.
\end{equation}


Here, our purpose is to find necessary and sufficient conditions for the posterior distribution to be proper for some general class of priors.  The proof of the following theorem will be left to the appendix.

\begin{theorem}\label{fundtheo} Let $\pi(\eta,\beta)$ be a class of priors such that
\begin{align*} \pi(\eta,\beta)\propto \exp\left(-p\beta^{-1}\right)\eta^r\beta^q \end{align*}
where $r$, $q$, $p$  are constants. Then:

\begin{itemize}
\item[(i)] If $r\neq -1$, then the posterior distribution of $(\beta,\eta)$ under the prior $\pi(\beta,\eta)$ is improper.
\item[(ii)] If $r=-1$, $p=0$ and at least two non-censored data are distinct, then the posterior distribution of $(\beta,\eta)$ under the prior $\pi(\beta,\eta)$ is proper if and only if $q>-m$.
\item[(iii)] If $r=-1$,  $p=0$ and $m\leq 1$ then the posterior distribution of $(\beta,\eta)$ under the prior $\pi(\beta,\eta)$ is improper.
\end{itemize}
\end{theorem}

\begin{corollary}\label{maintheorem2wei}
Let $\pi(\eta,\beta)$ be a class of improper priors such that
\begin{align*} \pi(\eta,\beta)\propto \eta^r\beta^q \end{align*}
where $r$ and $q$  are constants and suppose the posterior related to $\pi(\eta,\beta)$ is proper. Then the posterior moments relative to $\beta$ are finite, and the posterior moments relative to $\eta$ are not finite for both complete and censored data.
\end{corollary}

    \begin{proof}
Since the posterior is proper, by Theorem \ref{fundtheo} we have $r= -1$ and $m>-q$.

Now, given $k>0$ and $\pi^*(\eta,\beta)=\beta^{k}\pi(\eta,\beta)\propto \eta^{-1} \beta^{q+k}$ , since $m>-q>-(q+k)$, from Theorem \ref{fundtheo} it follows that the posterior relative to the prior $\pi^*(\eta,\beta)$ is proper and thus
\begin{equation*}
\begin{aligned}
E[\beta^{k}|\boldsymbol{x}]=\int_0^{\infty}\int_0^{\infty}\beta^{k}\pi(\eta,\beta) \eta^{m\beta}\beta^{m+q}\left(\prod_{i=1}^n x_i^{\delta_i\beta}\right)\exp \left\{-\sum_{i=1}^n x_i^\beta \eta^{\beta}\right\}d\eta d\beta<\infty.
\end{aligned}
\end{equation*}
Analogously, given $k>0$ it follows that $\eta^{k}\pi(\eta)\propto \eta^{k-1}\beta^{q}$, where $k-1>-1$ and thus from Theorem \ref{fundtheo} it follows that  
\begin{equation*}
\begin{aligned}
E[\eta^{k}|\boldsymbol{x}]=\int_0^{\infty}\int_0^{\infty}\eta^{k}\pi(\eta,\beta) \eta^{m\beta}\beta^{m+q}\left(\prod_{i=1}^n x_i^{\delta_i\beta}\right)\exp \left\{-\sum_{i=1}^n x_i^\beta \eta^{\beta}\right\}d\eta d\beta=\infty,
\end{aligned}
\end{equation*}
which concludes the proof.
\end{proof}
\section{Applications}

Here, we will apply our proposed methodology in different objective priors to verify if the posterior distributions are proper and verify if its posterior moments are finite. Once again in the following we let $m=\sum_{i=1}^{n}\delta_i$. 
Moreover, in the following, all considered priors shall take the form $\pi(\eta,\beta)=\exp\left(-p\beta^{-1}\right)\eta^{r}\beta^q$ where either $p=0$ or $r\neq -1$. Thus, due to Theorem \ref{fundtheo}, it follows that the posterior relative to such priors are improper in a case $m\leq 1$. Therefore, in the following propositions, we shall be concerned only with cases where $m\geq 2$, and under this hypothesis, we shall always suppose that there are at least two distinct non-censored data.

A common objective prior to that does not depend on any metric is obtained by considering independent uniform priors. In the case of positive parameters, the prior is given by $\pi_1\left(\eta,\beta\right)\propto 1$.

\begin{proposition} The posterior density using uniform prior is improper for all $m\geq 2$, in which case the posterior moments relative to $\beta$ are finite and the posterior moments relative to $\eta$ are not finite.
\end{proposition}
\begin{proof} Notice that given $\pi(\eta,\beta)=\pi_1(\eta,\beta)$, the hypothesis of Theorem \ref{fundtheo} hold for $r=0$, $q=0$ and $p=0$, and since $r\neq -1$, the conclusion of Theorem \ref{fundtheo} implies the posterior relative to $\pi_1(\eta,\beta)$ to be improper.
\end{proof}

Another common objective prior distribution can be obtained using the Jeffreys` first rule (see, Sun \cite{sun1997note}). This prior has invariance property under power transformations. As the parameters of the Weibull distribution are contained in the interval $(0,\infty)$, the prior distribution using the Jeffreys' rule is given by
\begin{equation}\label{priorrej}
\pi_2\left(\eta,\beta\right)\propto \frac{1}{\eta\beta} \cdot 
\end{equation}

\begin{proposition} The posterior density using the prior (\ref{priorrej}) is proper for all $m\geq 2$, in which case the posterior moments relative to $\beta$ are finite and the posterior moments relative to $\eta$ are not finite.
\end{proposition}
\begin{proof} If we let $\pi(\eta,\beta)=\pi_2(\eta,\beta)$, the hypothesis of Theorem \ref{fundtheo} are valid for $r=-1$, $q=-1$ and $p=0$, and thus applying Theorem \ref{fundtheo} it follows that the posterior relative obtained from the prior $\pi_2(\eta,\beta)$ is proper for all $m>-q=1$. The additional conclusions follows directly from Corollary \ref{maintheorem2wei}.
\end{proof}

The most well-known objective prior was introduced by Jeffreys \cite{jeffreys1946theory} which carries his name. The Jeffreys prior is obtained from the square root of the determinant of the Fisher information matrix $I(\eta,\beta)$, that is
\begin{equation}
\pi_3\left(\eta,\beta\right)\propto \sqrt{\f{det}I(\eta,\beta)}
\end{equation}
where $I(\eta,\beta)$ is the Fisher information matrix given by
\begin{equation}\label{mfisherIW}
I(\eta,\beta)=n
\begin{bmatrix}
\dfrac{\beta^2}{\eta^2} & \dfrac{\left(1-\gamma\right)}{\eta}  \\
\dfrac{\left(1-\gamma\right)}{\eta}  & \dfrac{1}{\beta^2}\left(\tfrac{\pi^2}{6} +\left(1-\gamma\right)^2 \right)
\end{bmatrix} ,
\end{equation}
and $\gamma\approx 0.5772156649$ is known as Euler-Mascheroni Constant.

This prior is widely used due to its invariance property under one-to-one
transformations of parameters. For the Weibull distribution, computing the square root of the determinant of $I(\eta,\beta)$  we have that
\begin{equation}\label{priorjIW}
\pi_3\left(\eta,\beta\right)\propto \frac{1}{\eta} \cdot
\end{equation}

\begin{theorem} The posterior density using the prior (\ref{priorjIW}) is proper for all $m\geq 2$, in which case the posterior moments relative to $\beta$ are finite and the posterior moments relative to $\eta$ are not finite.
\end{theorem} 
\begin{proof} By letting $\pi(\eta,\beta)=\pi_{J}(\eta,\beta)$, the hypothesis of Theorem \ref{fundtheo} if we assume $r=-1$, $q=-1$ and $p=0$. Thus, from the conclusion of Theorem \ref{fundtheo} it follows that the posterior relative to the prior $\pi_{J}(\eta,\beta)$ is proper for all $m>-q=0$. The additional conclusions follows directly from Corollary \ref{maintheorem2wei}.
\end{proof}

Zellner \cite{zellner1, zellner2} introduced an objective prior based on the Shannon entropy. The  Maximal Data Information Prior emphasis the information
in the likelihood function, therefore, its information is weak
in comparison with data information. This prior is obtained by
\begin{equation}\label{priorzIW1}
\pi_4\left(\eta,\beta\right)\propto \exp\left(H(\eta,\beta)\right),
\end{equation}
where $H(\eta,\beta)$ is the solution of the information measure given by
\begin{equation*}\label{ishanonIW1}
H(\eta,\beta)= \int_{0}^{\infty} \log\left[\beta\eta^\beta x^{-(\beta+1)}\exp\left\{-\left(\frac{\eta}{x}\right)^{\beta}\right\}\right]f\left(t|(\eta,\beta\right)dt.
\end{equation*}

The MDI prior has invariance
properties only for linear transformations on the parameters. For the Weibull distribution $H(\eta,\beta)$ can be written as:
\begin{equation*}\label{ishanonIW2}
H(\eta,\beta)=  \log(\eta\beta)+\gamma\left(1-\frac{1}{\beta}\right)-1 .
\end{equation*}

Therefore the MDI prior (\ref{priorzIW1}) for the Weibull distribution is given by
\begin{equation}\label{priorzIW2}
\pi_4\left(\eta,\beta\right)\propto \exp\left(-\gamma\beta^{-1} \right) \eta\beta.
\end{equation}

\begin{theorem} The posterior distribution obtained from the MDI prior (\ref{priorzIW2}) is improper for all $m\geq 2$.
\end{theorem} 
\begin{proof} Letting $\pi(\eta,\beta)=\pi_4(\eta,\beta)$,  notice the hypothesis of Theorem \ref{fundtheo} will be valid for  $r=1$, $q=1$ and $p=\gamma$. Thus, since $r\neq -1$, Theorem \ref{fundtheo} implies directly that the posterior relative to $\pi_1(\eta,\beta)$ is improper.
\end{proof}

An important non-informative prior was introduced by Bernardo \cite{bernardo1979a}, with further developments see for instance Bernardo \cite{bernardo2005} and Berger et al. \cite{berger2015}. His prior is referred to as reference prior, and it is obtained under the idea of maximizing the expected Kullback-Leibler divergence between the posterior distribution and the prior. The reference prior provides posterior distribution with interesting properties, such as invariance, consistent marginalization, and consistent sampling properties. 

The algorithm to derive the reference prior can be obtained in detail throughout Bernardo \cite{bernardo2005}. However, since the Fisher information $I(\eta,\beta)$ has a particular form, we will present one corollary that allows obtaining the reference priors quickly for the IW distribution.

\begin{corollary}\label{coro1} Considering that $\delta$ is the parameter of interest and $\lambda$ is the nuisance parameter, suppose the posterior
distribution of $(\delta,\lambda)$ is asymptotically normal with dispersion matrix $S(\delta,\lambda)=I^{-1}(\delta,\lambda)$. If the parameter space of $\delta$ is independent of $\lambda$ and the functions $S_{\delta,\delta}(\delta,\lambda)$ and $I_{\lambda,\lambda}(\delta,\lambda)$ can be factorize in the form
\begin{equation*}
S_{\delta,\delta}^{-\frac{1}{2}}(\delta,\lambda)=f_\theta(\theta)g_\theta(\lambda) \quad \mbox{and}  \quad  I_{\lambda,\lambda}^{\frac{1}{2}}(\theta,\lambda)=f_{\lambda}(\theta)g_{\lambda}(\lambda),
\end{equation*}
then the reference prior when $\delta$ is the parameter of interest and $\lambda$ is the nuisance parameter is simply $\pi_{\delta}(\delta,\lambda)=f_\delta(\delta)g_{\lambda}(\lambda)$ even if $\pi(\lambda|\delta)$ is not proper.
\end{corollary} 

Through the Corollary \ref{coro1} and using the Fisher information matrix (\ref{mfisherIW}) if $\eta$  is the parameter of interest and $\beta$ is the nuisance parameter, after some algebraic manipulations it follows that the reference prior  is the same as given by the Jeffreys' rule prior given by$ \pi_\eta\left(\eta,\beta\right)\propto \frac{1}{\eta\beta}$. Additionally, considering that $\beta$ is the parameter of interest and $\eta$ is the nuisance parameter, after some algebraic manipulations the reference prior has also the same form given by $\pi_\beta\left(\eta,\beta\right)\propto \frac{1}{\eta\beta}$. Since the reference priors have the same form of the Jeffreys' rule prior, it follows that the obtained posterior is proper for all $m\geq 2$, in which case the posterior moments relative to $\beta$ are finite and the posterior moments relative to $\eta$ are not finite.

\section{Other parametrization}

The results presented above can be easily extended to other commonly used parametrization. For instance, considering $\eta=1/\theta$, the PDF is given by
\begin{equation}\label{pdfgamma2}
f(x\,| \theta,\beta)=\beta\frac{x^{\beta-1}}{\theta^{\beta}}e^{- (\frac{x}{\theta})^\beta}.
\end{equation}
Therefore, for a prior distribution $\pi(\theta,\beta)$, the joint posterior distribution for $(\theta,\beta)$ is given by 
\begin{equation}\label{final2}
\pi(\theta,\beta|\boldsymbol{x})=\frac{\pi(\theta,\beta)}{d(\boldsymbol{x})} \theta^{-m\beta}\beta^{m}\prod_{i=1}^n x_i^{\delta_i\beta}\exp \left\{-\sum_{i=1}^n \left(\frac{x_i}{\theta}\right)^\beta \right\},
\end{equation}
where $m=\sum_{i=1}^{n}\delta_i$, and $d(\boldsymbol{x})$ is a normalized constant in the form 
\begin{equation}\label{posteriord22}
d(\boldsymbol{x})=\int_{0}^{\infty}\int_{0}^{\infty}\pi(\theta,\beta)\theta^{-m\beta}\beta^{m}\prod_{i=1}^nx_i^{\delta_i\beta}\exp \left\{-\sum_{i=1}^n \left(\frac{x_i}{\theta}\right)^\beta \right\}d\theta d\beta.
\end{equation}
In the following we consider the case where there are at least two distinct non-censored data. An analogous theorem can be proved for $m\leq 1$ by using item $(iii)$ of Theorem \ref{fundtheo}.

\begin{corollary}\label{fundtheo2} Let $\pi(\theta,\beta)$ be a class of improper priors such that
\begin{align*} \pi(\theta,\beta)\propto \exp\left(-p\beta^{-1}\right)\theta^r\beta^q \end{align*}
where $r$, $q$ and $p$ are constants and suppose that at least two non-censored data are distinct. Then the posterior distribution of $(\beta,\theta)$ under the prior $\pi(\beta,\theta)$ is improper in case $r\neq -1$, and is proper if and only if $r=-1$ and $q>-m$ in case $p=0$.
\end{corollary}
\begin{proof}
We have that 
\begin{equation*}
d(\boldsymbol{x})=\int_{0}^{\infty}\int_{0}^{\infty}\exp(-q\beta^{-1})\theta^{-m\beta+r}\beta^{m+q}\prod_{i=1}^nx_i^{\delta_i\beta}\exp \left\{-\sum_{i=1}^n \left(\frac{x_i}{\theta}\right)^\beta \right\}d\theta d\beta.
\end{equation*}
Thus, through the change of variables $\eta = \theta^{-1}$ it follows that
\begin{equation*}
d(\boldsymbol{x})=\int_{0}^{\infty}\int_{0}^{\infty}\pi_0(\eta,\beta)\eta^{m\beta}\beta^{m}\prod_{i=1}^nx_i^{\delta_i\beta}\exp \left\{-\sum_{i=1}^n \left(\eta x_i\right)^\beta \right\}d\theta d\beta
\end{equation*}
where $\pi_0(\eta,\beta)=\exp\left(-q\beta^{-1}\right)\eta^{-r+2}\beta^{q}$ and thus Theorem \ref{fundtheo} is applied using the prior $\pi_0(\eta,\beta)$ and implies in the proposed result.

\end{proof}

\begin{corollary}\label{maintheorem2wei2}
Let $\pi(\eta,\beta)$ be a class of improper priors such that
\begin{align*} \pi(\theta,\beta)\propto \theta^r\beta^q \end{align*}
where $r$ and $q$  are constants, and suppose the posterior related to $\pi(\theta,\beta)$ is proper. Then the posterior moments relative to $\beta$ are finite, and the posterior moments relative to $\theta$ are not finite.
\end{corollary}
\begin{proof} Just as in Theorem \ref{fundtheo2}, the proof is direct from Theorem \ref{maintheorem2wei} combined with an application of the change of variables $\eta = \theta^{-1}$ in the integral.
\end{proof}

\section{Discussion}

In this paper, we have studied the posterior properties of the Weibull distribution in the presence of complete and censored data. We proposed necessary and sufficient conditions to check if an improper prior distribution leads to a proper posterior. An exciting aspect of our finds is that one can check the posterior distribution is proper by analyzing only the behavior of the prior distribution. For proper posterior distributions, we also provided sufficient conditions to check if its higher posterior moments are finite or infinite.

The main theorem is used in different objective priors, such as the uniform prior, Jeffreys' first rule prior, Jeffreys prior, MDIP, and reference priors. We proved that among the considered priors, only the MDIP leads to an improper posterior. Although the priors above lead to proper posterior, we showed that none of the objective priors lead to finite posterior moments for the parameter $\eta$ and, therefore, should not be used as Bayes estimator. Further, the results were also presented for another standard parametrization, which can be used, for instance, in the objective prior obtained by Sun \cite{sun1997note}. Similar to the other parametrizations, the results showed that none of the objective priors lead to finite posterior moments for the parameter $\theta$. However, Sun \cite{sun1997note} used the posterior mean of $\theta$ in a real data example for complete data, but since we proved the posterior mean is not finite, it should not be used in practice. The find above is an interesting example where the MCMC methods may return a finite estimate when the true value is infinite. Hence, future researches should be careful in computing the posterior moments without checking its finiteness.

\section*{Acknowledgment}

Pedro L. Ramos is grateful to the S\~ao Paulo State Research Foundation (FAPESP Proc. 2017/25971-0). 

\section*{Reference}

\bibliography{referencias}

\section*{Appendix A - Preliminaries}

In the appendix, $\overline{\mathbb{R}}$ will denote the extended real number line $\R\cup\{-\infty,\infty\}$ and the subscript $*$ in $\mathbb{R}$ and $\overline{\mathbb{R}}$ will denote the exclusion of $0$ in these sets.

\begin{definition}\label{definition0} Let $\f{g}:\mathcal{U}\to\overline{\mathbb{R}}_*^+$ and $\f{h}:\mathcal{U}\to\overline{\mathbb{R}}_*^+$, where $\mathcal{U}\subset\mathbb{R}$. We say that $\f{g}(x)\propto \f{h}(x)$ if there exists $c_0\in \mathbb{R}^+_*$ and $c_1\in \mathbb{R}^+_*$ such that $c_0\f{h}(x) \leq \f{g}(x) \leq c_1\f{h}(x)$ for every $x\in \mathcal{U}$.
\end{definition}

\begin{definition}\label{definition1}
Let $a\in \mathbb{\overline{R}}$, $\f{g}:\mathcal{U}\to\mathbb{R^+}$ and $\f{h}:\mathcal{U}\to\mathbb{R^+}$, where $\mathcal{U}\subset\mathbb{R}$. We say that $\f{g}(x)\underset{x\to a}{\propto} \f{h}(x)$ if
\begin{equation*}
\liminf_{x\to a} \frac{\f{g}(x)}{\f{h}(x)} > 0\ \mbox{ and }\ \limsup_{x\to a} \frac{\f{g}(x)}{\f{h}(x)} < \infty  \,.
\end{equation*}
The meaning of the relations $\f{g}(x)\underset{x\to a^+}{\propto} \f{h}(x)$ and $\f{g}(x)\underset{x\to a^-}{\propto} \f{h}(x)$ for $a\in \mathbb{R}$ are defined analogously.
\end{definition}

Note that, from the above definiton, if for some $c\in \mathbb{R}^+_*$ we have that $\lim_{x\to a} \frac{\f{g}(x)}{\f{h}(x)} = c$, then it will follow that $\f{g}(x)\underset{x\to a}{\propto} \f{h}(x)$. The following proposition is a direct consequence of the above definition, where the functions $f_1$, $f_2$, $g_1$ and $g_2$ bellow are supposed to be positive in their domain $\mathcal{U}\subset \R$.

 The following propositions gives us a relation between Definition \ref{definition0} and Definition \ref{definition1} and its proofs can be seen in Ramos et al. \cite{ramos2017}.

\begin{proposition}\label{proportional1}
Let $\f{g}:(a,b)\to\mathbb{R^+}$ and $\f{h}:(a,b)\to\mathbb{R^+}$ be continuous functions on $(a,b)\subset\mathbb{R}$, where $a\in\overline{\mathbb{R}}$ and $b\in\overline{\mathbb{R}}$. Then $\f{g}(x)\propto \f{h}(x)$ if and only if $\f{g}(x)\underset{x\to a}{\propto} \f{h}(x)$ and $\f{g}(x)\underset{x\to b}{\propto} \f{h}(x)$.
\end{proposition}

\begin{proposition}\label{proportional2} Let $\f{g}:(a,b)\to\mathbb{R^+}$ and $\f{h}:(a,b)\to\mathbb{R^+}$ be continuous functions in $(a,b)\subset\mathbb{R}$, where $a\in\overline{\mathbb{R}}$ and $b\in\overline{\mathbb{R}}$, and let $c\in(a,b)$. Then, if either $\f{g}(x)\underset{x\to a}{\propto} \f{h}(x)$ or $\f{g}(x)\underset{x\to b}{\propto} \f{h}(x)$, it will follow respectively that
\begin{equation*}
\int_a^c g(x)\; dx \propto \int_a^c h(x)\; dx\ \mbox{ or }\ \int_c^b g(x)\; dx \propto \int_c^b h(x)\; dx \,.
\end{equation*}
\end{proposition} 

\section*{Appendix B - Proof of Theorem 2.1}

During this proof we shall use the fact that $\int_0^\infty x^{a-1}\exp\left(-b x\right)\, dx <\infty$ if and only if $a>0$, which can be easily proved for $a\leq 0$ by using Proposition \ref{proportional1} and the fact that $\exp\left(-b u\right)\underset{\alpha\to 0}{\propto} 1$, and for $a>0$ by the change of variables $bx=y$ in the integral and the definition of Gamma function.

Before we prove the items we develop the integrals used. From the hypothesis it follows that
\begin{align*}
\pi(\eta,\beta|\boldsymbol{x})& =  \pi(\eta,\beta)\eta^{m\beta}\beta^{m}\left(\prod_{i=1}^n x_i^{\delta_i\beta}\right)\exp \left\{-\sum_{i=1}^n x_i^\beta \eta^{\beta}\right\} \\
& \propto \exp\left(-p\beta^{-1}\right)\eta^{m\beta+r}\beta^{m+q}\left(\prod_{i=1}^n x_i^{\delta_i\beta}\right)\exp \left\{-\sum_{i=1}^n x_i^\beta \eta^{\beta}\right\}.
\end{align*}
Thus, from the change of variable $u=\sum_{i=1}^n x_i^{\beta} \eta^{\beta}$ in the integral it follows that
\begin{align*}
d(x)&\propto \int\limits_0^\infty\int\limits_0^{\infty} \exp\left(-p\beta^{-1}\right)\beta^{m+q}\left(\prod_{i=1}^n x_i^{\delta_i\beta}\right)\eta^{m\beta+r}\exp \left\{-x_i^\beta\eta^{\beta}\right\} \, d\eta  \, d\beta,\\
&=  \int\limits_0^\infty \exp\left(-p\beta^{-1}\right) \beta^{m+q-1}\dfrac{\left(\prod_{i=1}^n{x_i^{\delta_i \beta}}\right)}{\left(\sum_{i=1}x_i^\beta\right)^{m+\frac{r+1}{\beta}}}\int\limits_0^{\infty}u^{m-1+\frac{r+1}{\beta}}\exp\{-u\} \, du  \, d\beta
\end{align*}

On the other hand, if we denote $x_{\max}=\max_{1\leq i\leq n} x_i$ and if $k\in \N$ correspond to the number of indexes $i\in\{1,\cdots,n\}$ such that $x_i=x_{\max}$ then, since $\lim_{\beta\to \infty} \frac{x_i^\beta}{x_{\max}^{\beta}} = 0$ for all $x_i< x_{\max}$ it follows that
\begin{equation*}\lim_{\beta\to \infty} \frac{\left(\sum_{i=1}^n x_i^\beta\right)}{ (x_{\max})^{\beta}} = k \Rightarrow \sum_{i=1}^n x_i^\beta\underset{\beta\to \infty}{\propto} (x_{\max})^{\beta},
\end{equation*}
and, moreover, trivially it follows that
\begin{equation*}\lim_{\beta\to 0} \frac{\left(\sum_{i=1}^n x_i^\beta\right)}{(x_{\max})^{\beta}}=1 \Rightarrow \sum_{i=1}^n x_i^\beta\underset{\beta\to 0}{\propto} (x_{\max})^{\beta}.
\end{equation*}
Thus, from Proposition \ref{proportional1} it follows that $\left(\sum_{i=1}^n x_i^\beta\right)\propto (x_{\max})^{\beta}$ in $[0,\infty)$ and therefore
\begin{equation}\label{important}
d(\boldsymbol{x}) \propto  \int\limits_0^\infty \exp\left(-p\beta^{-1}\right) \beta^{m+q-1}\dfrac{\left(\prod_{i=1}^n{x_i^{\delta_i \beta}}\right)}{(x_{\max})^{m\beta}}\int\limits_0^{\infty}u^{m-1+\frac{r+1}{\beta}}\exp\{-u\} \, du  \, d\beta.
\end{equation}\vspace{0.3cm}

\noindent Proof of item $(i)$: We divide the proof in the cases where $r<-1$ and $r>-1$.

First let us suppose that $r<-1$ and $m=0$. Then it follows from (\ref{important}) that
\begin{equation*}
d(\boldsymbol{x}) \propto  \int\limits_0^\infty \exp\left(-p\beta^{-1}\right) \beta^{q-1}\int\limits_0^{\infty}u^{-1+\frac{r+1}{\beta}}\exp\{-u\} \, du  \, d\beta.
\end{equation*}
and since $\frac{r+1}{\beta}<0$ for all $\beta>0$ it follows that $\int\limits_0^{\infty}u^{-1+\frac{r+1}{\beta}}\exp\{-u\} \, du=\infty$ for all $\beta>0$ and thus $d(x)=\infty$ in this case.

Now let us suppose that $r<-1$ and $m>0$. Therefore, for every $\beta\in \left(0,-\frac{r+1}{m}\right]$ it follows that $m+\frac{r+1}{\beta}<0$ and thus
\begin{equation*}
\begin{aligned}
\exp\left(-p\beta^{-1}\right) \beta^{m+q-1}\dfrac{\left(\prod_{i=1}^n{x_i^{\delta_i \beta}}\right)}{(x_{\max})^{m\beta}}\int\limits_0^\infty u^{m-1+\frac{r+1}{\beta}}\exp\{-u\}\, d\eta= \infty,
\end{aligned}
\end{equation*}
for all $\beta\in (0,\frac{r+1}{m}]$. Therefore it follows from (\ref{important}) that
\begin{equation*}
\begin{aligned}
d(\boldsymbol{x}) \propto  &\int\limits_0^\infty \exp\left(-p\beta^{-1}\right) \beta^{m+q-1}\dfrac{\left(\prod_{i=1}^n{x_i^{\delta_i \beta}}\right)}{(x_{\max})^{m\beta}}\int\limits_0^{\infty}u^{m-1+\frac{r+1}{\beta}}\exp\{-u\} \, du  \, d\beta \\ 
\geq &\int\limits_0^{\frac{r+1}{m}} \exp\left(-p\beta^{-1}\right) \beta^{m+q-1}\dfrac{\left(\prod_{i=1}^n{x_i^{\delta_i \beta}}\right)}{(x_{\max})^{m\beta}} \int\limits_0^{\infty}u^{m-1+\frac{r+1}{\beta}}\exp\{-u\} \, du  \, d\beta  = \infty,
\end{aligned}
\end{equation*}
and hence $d(\boldsymbol{x})=\infty$ in case $r<-1.$

Now let us suppose that $r> -1$. Therefore, from (\ref{important}) and letting $h=\log\left(\dfrac{\left(x_{\max}\right)^m}{\prod_{i=1}^n{x_i^{\delta_i}}}\right)\geq 0$ it follows that
\begin{equation}\label{important2}
d(\boldsymbol{x}) \propto  \int\limits_0^\infty \exp\left(-p\beta^{-1}\right)\beta^{m+q-1}\exp\left(-h\beta\right)\Gamma\left(m+\frac{r+1}{\beta}\right)  \, d\beta
\end{equation}
Thus, using the change of variables $\alpha = m+\frac{r+1}{\beta}$ in the integral on (\ref{important2}), and using that $\exp\left(\frac{p(\alpha-m)}{r+1}\right)= \exp\left(\frac{-pm}{r+1}\right)\exp\left(\frac{p\alpha}{r+1}\right) \propto \exp\left(\frac{p\alpha}{r+1}\right)$ it follows that
\begin{equation}\label{final}
d(\boldsymbol{x}) \propto  \int\limits_m^\infty \frac{\Gamma\left(\alpha\right)}{\exp\left(\frac{p\alpha}{r+1}\right)(\alpha-m)^{m+q+1}\exp\left(\frac{h(r+1)}{\alpha-m}\right)}  \, d\alpha .
\end{equation}
On the other hand, from the Stirling formula and since $\lim_{\alpha\to \infty} \exp\left(\frac{h(r+1)}{\alpha-m}\right)=1$ and $\lim_{\alpha\to \infty}\frac{(a-m)^{m+q+1}}{\alpha^{m+q+1}}=1$, it follows that

\begin{align*}\lim_{\alpha\to \infty} \frac{\Gamma\left(\alpha\right)}{\exp\left(\frac{p\alpha}{r+1}\right)(\alpha-m)^{m+q+1}\exp\left(\frac{h(r+1)}{\alpha-m}\right)} = \sqrt{2\pi}\lim_{\alpha\to \infty} \frac{\alpha^{\alpha-\frac{1}{2}}\exp(-\alpha)}{\exp\left(\frac{p\alpha}{r+1}\right)\alpha^{m+q+1}}\\
=\sqrt{2\pi}\lim_{\alpha\to \infty}\exp\left(\alpha\ln(\alpha)-\frac{1}{2}\ln(\alpha)-\alpha-\frac{p\alpha}{r+1}-\ln(\alpha)\left(m+q+1\right)\right)\\
=\sqrt{2\pi}\lim_{\alpha\to \infty}\exp\left(\alpha\left(\ln(\alpha)-\left(1+\frac{p}{r+1}\right)-\frac{\ln(\alpha)}{\alpha}\left(m+q+\frac{3}{2}\right)\right) \right)=\infty,
\end{align*}
where the last inequality follows directly from $\lim_{\alpha\to\infty}\frac{\ln(\alpha)}{\alpha}=0$ and $\lim_{\alpha\to \infty}\ln(\alpha)=\infty$. Therefore from (\ref{final}) we conclude that $d(x) = \infty$ in case $r>-1$, and thus item $(i)$ is proved.\vspace{0.3cm}

Proof of item $(ii)$: If we suppose that $r=-1$ and $p=0$ then, once again from (\ref{important}) and letting $h=\log\left(\dfrac{\left(x_{\max}\right)^m}{\prod_{i=1}^n{x_i^{\delta_i}}}\right)\geq 0$, it follows that
\begin{equation*}
\begin{aligned}
d(\boldsymbol{x}) \propto\int\limits_0^\infty  \beta^{m+q-1}\exp\left(-h\beta\right)\Gamma(m)\, d\beta\propto \int\limits_0^\infty  \beta^{m+q-1}\exp\left(-h\beta\right)\, d\beta,
\end{aligned}
\end{equation*}
and since by hypothesis of item $(ii)$ there are at least two distinct non-censored data, it follows that $h=\log\left(\dfrac{\left(x_{\max}\right)^m}{\prod_{i=1}^n x_i^{\delta_i}}\right)>0$ and thus by the last proportionality it follows that $d(\boldsymbol{x})<\infty$ if and only if $m+q>0$, which concludes the proof of item $(ii)$.\vspace{0.3cm}

Proof of item $(iii)$:  If we suppose that $r=-1$, $p=0$ and $m= 1$, then from (\ref{important}) it follows that
\begin{equation*}
\begin{aligned}
d(\boldsymbol{x}) \propto\int\limits_0^\infty  \beta^{q-1}\int\limits_0^\infty \exp\left\{-u\right\}\, du\, d\beta=  \int\limits_0^\infty  \beta^{q-1}\, d\beta,
\end{aligned}
\end{equation*}
and thus $d(x)=\infty$ in this case.

Finally if we suppose that $r=-1$, $p=0$ and $m=0$ then from (\ref{important}) it follows directly that
\begin{equation}
d(\boldsymbol{x}) \propto  \int\limits_0^\infty \beta^{q-1}\int\limits_0^{\infty}u^{-1}\exp\{-u\} \, du  \, d\beta,
\end{equation}
and since $\int_0^\infty u^{-1}\exp(-u)\, du=\infty$, it follows that $d(x)=\infty$, which completes the proof.

\end{document}